\documentclass[12pt]{article}

\usepackage{amsmath,amssymb,amsfonts,amsthm}
\usepackage{ifpdf}
\usepackage{enumerate}
\usepackage{ulem}
\usepackage{leftidx}

\usepackage{breqn}

\usepackage{tikz}
\usetikzlibrary{calc}
\usepackage{tikz-qtree}

\usepackage[plainpages=false,hypertexnames=false,pdfpagelabels]{hyperref}

\newcommand{\arxiv}[1]{\href{http://arxiv.org/abs/#1}{\tt arXiv:\nolinkurl{#1}}}
\newcommand{\doi}[1]{\href{http://dx.doi.org/#1}{{\tt DOI:#1}}}

\newcommand{\googlebooks}[1]{(preview at \href{http://books.google.com/books?id=#1}{google books})}

\theoremstyle{plain}
\newtheorem{prop}{Proposition}[section]

\newtheorem{thm}[prop]{Theorem}

\newtheorem{lem}[prop]{Lemma}

\newtheorem*{cor*}{Corollary}

\newtheorem{fact}{Fact}

\numberwithin{equation}{section}

\theoremstyle{remark}

\newtheorem*{rem*}{Remark}               
\newtheorem*{ex*}{Example}                

\theoremstyle{definition}

\newtheorem*{defn*}{Definition}             

\theoremstyle{plain}

\newcounter{comment}
\newcommand{\noop}[1]{}

\def\clap#1{\hbox to 0pt{\hss#1\hss}}

\def\semicolon{;}
\def\applytolist#1{
    \expandafter\def\csname multi#1\endcsname##1{
        \def\multiack{##1}\ifx\multiack\semicolon
            \def\next{\relax}
        \else
            \csname #1\endcsname{##1}
            \def\next{\csname multi#1\endcsname}
        \fi
        \next}
    \csname multi#1\endcsname}

\def\calc#1{\expandafter\def\csname c#1\endcsname{{\mathcal #1}}}
\applytolist{calc}QWERTYUIOPLKJHGFDSAZXCVBNM;
\def\bbc#1{\expandafter\def\csname bb#1\endcsname{{\mathbb #1}}}
\applytolist{bbc}QWERTYUIOPLKJHGFDSAZXCVBNM;
\def\bfc#1{\expandafter\def\csname bf#1\endcsname{{\mathbf #1}}}
\applytolist{bfc}QWERTYUIOPLKJHGFDSAZXCVBNM;

\renewcommand{\imath}{\mathfrak{i}}
\renewcommand{\jmath}{\mathfrak{j}}

\newcommand{\iso}{\cong}

\makeatletter
\newcommand{\hashdef}[2]{\@namedef{#1}{#2}}
\newcommand{\hashlookup}[1]{\@nameuse{#1}}
\makeatother

\IfFileExists{../../graphs/lookup.tex}%
	{\newcommand{\pathtographs}{../../graphs/}}%
	{\newcommand{\pathtographs}{diagrams/graphs/}}

\hashdef{bwd1v1v1v1v1v1v1v1p1v1x0p0x1v1x0p0x1duals1v1v1v1v1x2v2x1}{C53E32050C862F37}%
\hashdef{bwd1v1v1v1v1v1v1v1p1v1x0p1x0duals1v1v1v1v1x2}{CA1952BED6C8EB18}%
\hashdef{bwd1v4v1duals1v1}{E69D51A0185FD82C}%

\newcommand{\bigraph}[1]{{\hspace{-3pt}\begin{array}{c}%
  \raisebox{-2.5pt}{\includegraphics[height=6mm]{\pathtographs \hashlookup{#1}}}%
\end{array}\hspace{-3pt}}}

\newcommand{\Hom}{\operatorname{Hom}}
\newcommand{\End}{\operatorname{End}}
\newcommand{\Rep}{\operatorname{Rep}}

\title{The centre of the extended Haagerup subfactor has 22 simple objects}
\author{Scott Morrison and Kevin Walker}

\begin{document}
\maketitle

We explain a technique for discovering the number of simple objects in $Z(\cC)$, the center of a fusion category $\cC$, as well as the combinatorial data of the induction and restriction functors at the level of Grothendieck rings. The only input is the fusion ring $K(\cC)$ and the dimension function $K(\cC) \to \mathbb{C}$.

The method is not guaranteed to succeed (it may give spurious answers besides the correct one, or it may simply take too much computer time), but it seems it often does. We illustrate by showing that there are 22 simple objects in the center of the extended Haagerup subfactor \cite{0909.4099}, and that the induction functors from the $6$-object and $8$-object fusion categories arising as the even parts of the subfactor are given by
\begin{align*}
I_{EH1} & = \left(
\begin{array}{cccccccccccccccccccccc}
 1 & 1 & 1 & 1 & 1 & 1 & 0 & 0 & 0 & 0 & 0 & 0 & 0 & 0 & 0 & 0 & 0 & 0 & 0 & 0 & 0 & 0 \\
 2 & 1 & 1 & 1 & 0 & 0 & 1 & 1 & 1 & 1 & 1 & 1 & 1 & 1 & 0 & 0 & 0 & 0 & 1 & 1 & 1 & 1 \\
 1 & 2 & 1 & 1 & 0 & 3 & 2 & 2 & 2 & 2 & 2 & 2 & 2 & 2 & 1 & 1 & 1 & 1 & 1 & 1 & 1 & 1 \\
 2 & 4 & 4 & 1 & 0 & 2 & 4 & 4 & 4 & 4 & 3 & 3 & 3 & 3 & 3 & 3 & 3 & 3 & 1 & 1 & 1 & 1 \\
 4 & 5 & 2 & 1 & 0 & 3 & 5 & 5 & 5 & 5 & 3 & 3 & 3 & 3 & 4 & 4 & 4 & 4 & 1 & 1 & 1 & 1 \\
 1 & 3 & 2 & 1 & 0 & 2 & 3 & 3 & 3 & 3 & 1 & 1 & 1 & 1 & 2 & 2 & 2 & 2 & 1 & 1 & 1 & 1
\end{array}
\right) \\
I_{EH2} & = \left(
\begin{array}{cccccccccccccccccccccc}
 2 & 1 & 1 & 1 & 1 & 0 & 0 & 0 & 0 & 0 & 0 & 0 & 0 & 0 & 0 & 0 & 0 & 0 & 0 & 0 & 0 & 0 \\
 1 & 1 & 1 & 1 & 0 & 1 & 1 & 1 & 1 & 1 & 1 & 1 & 1 & 1 & 0 & 0 & 0 & 0 & 1 & 1 & 1 & 1 \\
 2 & 2 & 1 & 1 & 0 & 2 & 2 & 2 & 2 & 2 & 2 & 2 & 2 & 2 & 1 & 1 & 1 & 1 & 1 & 1 & 1 & 1 \\
 1 & 4 & 4 & 1 & 0 & 3 & 4 & 4 & 4 & 4 & 3 & 3 & 3 & 3 & 3 & 3 & 3 & 3 & 1 & 1 & 1 & 1 \\
 3 & 4 & 2 & 1 & 0 & 2 & 4 & 4 & 4 & 4 & 2 & 2 & 2 & 2 & 3 & 3 & 3 & 3 & 1 & 1 & 1 & 1 \\
 3 & 4 & 2 & 1 & 0 & 2 & 4 & 4 & 4 & 4 & 2 & 2 & 2 & 2 & 3 & 3 & 3 & 3 & 1 & 1 & 1 & 1 \\
 1 & 1 & 0 & 0 & 0 & 1 & 1 & 1 & 1 & 1 & 1 & 1 & 1 & 1 & 1 & 1 & 1 & 1 & 0 & 0 & 0 & 0 \\
 1 & 1 & 0 & 0 & 0 & 1 & 1 & 1 & 1 & 1 & 1 & 1 & 1 & 1 & 1 & 1 & 1 & 1 & 0 & 0 & 0 & 0
\end{array}
\right)
\end{align*}

Observe that the fifth column corresponds to $\mathbf{1} \in Z(\cC)$, and that there are four sets of four objects in $Z(\cC)$ which each restrict the same way to both $EH1$ and $EH2$.

Of course, it would also be interesting to compute all of $K(Z(\cC))$, and even more interesting to compute the $S$ and $T$ matrices. We don't address those questions here.

\subsection*{Acknowledgements}
We'd like to thank Brendan McKay for suggesting the reduction used in Theorem \ref{thm:reduce}, and Pinhas Grossman and Noah Snyder for interesting discussions while we were doing this work. Scott Morrison was supported by an Australian Research Council `Discovery Early Career Researcher Award', DE120100232 and `Discovery Project' DP140100732. Scott Morrison and Kevin Walker were supported by DOD-DARPA grant HR0011-12-1-0009. We would like to thank the Erwin Schr\"odinger Institute and its 2014 programme on ``Modern Trends in Topological Quantum Field Theory'' for their hospitality.

\section{Fusion categories and centers}
Given a fusion 2-category $\cC$, we use the following facts about its center, the modular tensor category $Z(\cC)$.

\begin{fact}
For each simple object $X \in Z(\cC)$, $\dim(X)$ divides $\dim(\cC)$ as an algebraic integer. 
\end{fact}
(This follows from Lemma 1.2 of \cite{MR1617921} and $\dim(Z(\cC)) = \dim(\cC)^2$.)

For a fusion 2-category $\dim(\cC)$ can be computed as the sum of $\dim(X)^2$ over $X$ in the collection of simple endomorphisms of any chosen object.

\begin{fact}
For each simple object $X \in Z(\cC)$, $\dim(X)$ is a $d$-number in the sense of Ostrik. \cite{0810.3242}
\end{fact}

Recall that an algebraic integer with minimal polynomial $p(x)=\sum_{i=0}^n a_i x^i$ is a $d$-number if $a_0^i$ divides $a_{n-i}^n$ for each $i$.

\begin{fact}
For each object $a \in \cC$, there is an induction functor $I: \End_\cC(a) \to Z(\cC)$ which is a pivotal functor. In particular, it induces a ring homomorphism $K(\End_\cC(a)) \to K(Z(\cC))$, and it preserves dimensions.
\end{fact}

The center may be realized as $\Rep \cC(S^1)$, the representation category of the annular category of $\cC$, and the induction functor is given by the inclusion of the rectangle in the annulus.

\begin{fact}
Given $X \in \End_\cC(a)$ and $Y \in \End_\cC(b)$,
the space $\Hom_{Z(\cC)}(I(X), I(Y))$ has a basis
$$
\begin{tikzpicture}
\draw[ultra thick] (0,0) circle (2cm);
\draw[ultra thick] (0,0) circle (1cm);
\draw (0,0) circle (1.5cm);
\draw[->] (1.5,0) -- (1.5,-0.001);
\draw[->] (80:1.5) -- (79.99:1.5);
\node at (90:1.7cm) {$Z$};
\draw (60:1.5) node[above] {$v$} -- +(0,-0.6) node[below left=-2pt] {$X$};
\draw[->] ($(60:1.5)+(0,-0.3)$) -- +(0,0.001);
\draw (120:1.5) node[below] {$u$}  -- +(0,0.58) node[above left=-2pt] {$Y$}; 
\draw[->] ($(120:1.5)+(0,0.29)$) -- +(0,0.001);
\node at (-90:1.7cm) {$W$};
\end{tikzpicture}
$$
where $W,Z \in \Hom_\cC(b,a)$ and $u$ runs over a basis of $\Hom^2_\cC(W, Y \otimes Z)$ while $v$ runs over a basis of $\Hom^2_\cC(Z \otimes X,W)$.

In particular,
\begin{align}
\label{eq:dim}
\dim \Hom_{(Z(\cC)} (I(X), I(Y)) & = \notag\\ \sum_{W,Z \in \Hom_\cC(b,a)}  \dim \Hom^2_\cC&(W, Y \otimes Z) \dim \Hom^2_\cC(Z \otimes X,W).
\end{align}
\end{fact}
(This follows from the apparently folkloric spine lemma, c.f. \cite{spine-lemma-MO-question}.)

Equivalent to this fact is that the composition of induction from $\End_\cC(a)$ to $Z(\cC)$ followed by restriction from $Z(\cC)$ to $\End_\cC(b)$ is given on objects by $X \mapsto \oplus_V V^*  X  V$, where the sum is over simple objects $V \in \Hom_\cC(a,b)$. (See, e.g. \cite[Proposition 5.4]{MR2183279}.)

With respect to the basis of simples, $I_a: K(\End_\cC(a)) \to K(Z(\cC))$ has a matrix $A_a$, with rows indexed by simples in $\End_\cC(a)$ and columns indexed by simples in $Z(\cC)$.
We order the simples in $\End_\cC(a)$ by dimension. We order the simples in $Z(\cC)$ so that the columns of $A_a$ appear in reverse lexicographic order.

We denote by $M_{ab}$ the matrix whose $ij$ entry is $\dim \Hom_{(Z(\cC)} (I(X), I(Y))$ computed as in Equation \eqref{eq:dim}, for $X$ the $i$-th simple in $\End_\cC(a)$ and $Y$ the $j$-th simple in $\End_\cC(b)$. Equation \eqref{eq:dim} tells us that $M_{ab} = A_a A_b^t$.
We denote by $\cM$ the block matrix whose $ab$ block is $M_{ab}$, and $\cA$ the matrix made by stacking the matrices $A_a$ above each other. Then $\cM = \cA \cA^t$.

Our task now is to compute all possible forms for the matrix $\cA$.

\section{Combinatorics}
We begin with a symmetric $n$-by-$n$ matrix $\cM$ with non-negative integer entries.
A \textbf{decomposition} of $\cM$ is a $n$-by-$m$ matrix $\cA$ (for some $m$) with non-negative integer entries, so $\cM = \cA \cA^t$.

Let $d$ be some algebraic number. Fix some collection of vectors $v_i \in \mathbb{Q}(d)^n$.
Further fix an algebraic number $\cD \in \mathbb{Q}(d)$. We wish to find all $n$-by-$m$ matrices $\cA$ so that $\cM = \cA \cA^t$, and for each column $w$ of $\cA$ and each $i$, $v_i.w$ is an Ostrik $d$-number and divides $\cD$ as an algebraic integer. We call such a decomposition for $(\cM, \{v_i\}, \cD)$ an \textbf{algebraic decomposition}.

Because we work in the fixed number field $\mathbb{Q}(d)$, we can easily compute the minimal polynomial of $w.v_i$, and hence determine if it is an Ostrik $d$-number. 

For each object $a$ of $\cC$, take $v_a$ to be the vector of dimensions of simple endomorphisms of $a$.
If there are $n_a$ simple endomorphisms of $a$, $v_a \in \mathbb{Q}(d)^{n_a}$. We abuse notation and also think of $v_a$ as the corresponding vector in $\mathbb{Q}(d)^n \iso \oplus_a\mathbb{Q}(d)^{n_a}$ by padding with zeroes.

We can summarize the facts from the previous section as
\begin{thm} Let $\cC$ be a fusion 2-category, 
 $\cM$ be the matrix of inner products defined by Equation \eqref{eq:dim}, and $\cA$ be the induction matrix defined above.

Then $\cA$ is an algebraic decomposition of $(\cM, \{v_a\}, \dim(\cC))$.
\end{thm}
Observe that we can take $d$ to be the Frobenius-Perron eigenvalue of $\cM$, and in fact $d = k \cD$, where $k$ is the number of simple objects of $\cC$.

Since a decomposition $\cA$ being algebraic implies strong conditions on the columns, we intend to enumerate algebraic decompositions by building up the matrix column at a time. First however, we perform a reduction of the problem (replacing $\cM$ above with another, smaller $\cM'$) which will be essential for reasonable runtimes on intended examples.

\begin{thm}
\label{thm:reduce}
Suppose $\cM$ has rank $r$, and the top left $r$-by-$r$ minor $\cM'$ is nonsingular. There is a unique $n$-by-$r$ matrix $\cR$ (with rational entries), so  $\cM = \cR \cM' \cR^t$, and the (not necessarily algebraic) decompositions of $\cM$ are exactly those $\cR \cA'$, for $\cA'$ is a decomposition of $\cM'$, where $\cR \cA'$ has non-negative integer entries.

Moreover, for any collection of vectors $\{v_i\}\subset \mathbb{Q}(d)^n$, the algebraic decompositions of $(\cM, \{v_i\}, \cD)$ correspond in the same way to those algebraic decompositions of $(\cM', \{  v_i \cR \}, \cD)$ with non-negative integer entries.
\end{thm}
\begin{proof}
Certainly for any decomposition $\cM = \cA \cA^t$ taking $\cA'$ to be the first $r$ rows of $\cA$ gives a decomposition $\cM' = \cA' {\cA'}^t$. In the opposite direction, first note that there is some $n$-by-$r$ matrix $\cR$ (with rational entries) so that $\cM = \cR \cM' \cR^t$. 
Now given a decomposition $\cM' = \cA' {\cA'}^t$, we obtain a not necessarily integral decomposition $\cM = (\cR \cA') (\cR \cA')^t$. Finally, if we obtained $\cA'$ as the first $r$ rows of an $\cA$ satisfying $\cM = \cA \cA^t$, this reconstructs the original $\cA$.

Thus we see that it suffices to search for decompositions $\cM' = \cA' {\cA'}^t$, and take exactly those $\cR \cA'$ which are integral.

For the last part, we see that a column $w$ of $\cA'$ satisfies the algebraic conditions with respect to $\{v_i \cR \}$ exactly if $\cR w$ (the corresponding column of $\cR \cA'$) satisfies the algebraic conditions with respect to $\{v_i\}$, since $v_i.(\cR w) = (v_i \cR).w$.
\end{proof}

A \textbf{partial algebraic decomposition} of $(\cM, \{v_i\}, \cD)$ is a matrix $\cB$ so that $\cM - \cB \cB^t$ is a non-negative matrix, and the columns of $\cB$ satisfy the same conditions, determined by $\{v_i\}$ and $\cD$, as the columns of an algebraic decomposition.

In particular an algebraic decomposition is a partial algebraic decomposition.

\begin{lem}
Deleting a column from a partial algebraic decomposition gives another partial algebraic decomposition.
\end{lem}

We say a \textbf{new column} for a partial algebraic decomposition $\cB$ is a vector $w$, such that
\begin{enumerate}
\item each $w_i \geq 0$,
\item if $p$ is the greatest number such that the top left $p$-by-$p$ minor of $\cM - \cB \cB^t$ is exactly zero, $w_{p+1} > 0$,
\item writing $u$ for the last column of $\cB$,
if the first $k$ entries of the $u$ agree with the first $k$ entries of $w$, then $w_{k+1} \leq u_{k+1}$,
\item $w_i \leq (\cM - \cB \cB^t)_{ij} / w_j$, for each $j \leq i$, 
\item $w$ satisfies the algebraic conditions determined by $\{v_i\}$ and $\cD$, and
\item $\cM - \cB \cB^t - w w^t$ is a non-negative matrix.
\end{enumerate}
We can clearly enumerate all possible new columns for $\cB$; in practice for the last condition, we numerically estimate all the eigenvalues, accepting $w$ if they are all at least $-0.001$.
Condition (4) is redundant with (6); we include it as a token optimization.

We now have
\begin{thm}
Every partial algebraic decomposition with $k$ columns in reverse lexicographic order may be obtained by appending a new column to some partial algebraic decomposition with $k-1$ columns in reverse lexicographic order.
\end{thm}

This theorem gives a relatively efficient mechanism for enumerating all algebraic decompositions of a given $(\cM, \{v_i\}, \cD)$. It is implemented in a {\tt Mathematica} notebook available with the {\tt arXiv} sources of this article. That notebook relies on the {\tt FusionAtlas} package introduced in \cite{index5-part1,index5-part2,index5-part3,index5-part4}, although only to prepare the fusion rings of (and calculate the dimesions for) some familiar examples. It should be easy to see how to run it without this dependency.

\section{Calculations}
For the extended Haagerup principal graphs
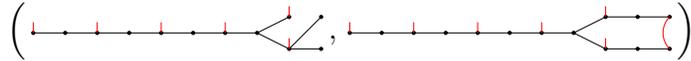
\begin{equation*}
\left( \bigraph{bwd1v1v1v1v1v1v1v1p1v1x0p1x0duals1v1v1v1v1x2}, \bigraph{bwd1v1v1v1v1v1v1v1p1v1x0p0x1v1x0p0x1duals1v1v1v1v1x2v2x1}\right)
\end{equation*}
there are a unique fusion rings for the two even parts, given by
{
\setlength\arraycolsep{3pt}
\scriptsize
\begin{align*}
&\left(
\begin{array}{cccccc}
 1 & 0 & 0 & 0 & 0 & 0 \\
 0 & 1 & 0 & 0 & 0 & 0 \\
 0 & 0 & 1 & 0 & 0 & 0 \\
 0 & 0 & 0 & 1 & 0 & 0 \\
 0 & 0 & 0 & 0 & 1 & 0 \\
 0 & 0 & 0 & 0 & 0 & 1 \\
\end{array}
\right), 
\left(
\begin{array}{cccccc}
 0 & 1 & 0 & 0 & 0 & 0 \\
 1 & 1 & 1 & 0 & 0 & 0 \\
 0 & 1 & 1 & 1 & 0 & 0 \\
 0 & 0 & 1 & 1 & 1 & 1 \\
 0 & 0 & 0 & 1 & 2 & 1 \\
 0 & 0 & 0 & 1 & 1 & 0 \\
\end{array}
\right), 
\left(
\begin{array}{cccccc}
 0 & 0 & 1 & 0 & 0 & 0 \\
 0 & 1 & 1 & 1 & 0 & 0 \\
 1 & 1 & 1 & 1 & 1 & 1 \\
 0 & 1 & 1 & 2 & 3 & 1 \\
 0 & 0 & 1 & 3 & 3 & 2 \\
 0 & 0 & 1 & 1 & 2 & 1 \\
\end{array}
\right),  \\ &\qquad\qquad
\left(
\begin{array}{cccccc}
 0 & 0 & 0 & 1 & 0 & 0 \\
 0 & 0 & 1 & 1 & 1 & 1 \\
 0 & 1 & 1 & 2 & 3 & 1 \\
 1 & 1 & 2 & 4 & 5 & 3 \\
 0 & 1 & 3 & 5 & 6 & 3 \\
 0 & 1 & 1 & 3 & 3 & 2 \\
\end{array}
\right), 
\left(
\begin{array}{cccccc}
 0 & 0 & 0 & 0 & 1 & 0 \\
 0 & 0 & 0 & 1 & 2 & 1 \\
 0 & 0 & 1 & 3 & 3 & 2 \\
 0 & 1 & 3 & 5 & 6 & 3 \\
 1 & 2 & 3 & 6 & 7 & 4 \\
 0 & 1 & 2 & 3 & 4 & 2 \\
\end{array}
\right), 
\left(
\begin{array}{cccccc}
 0 & 0 & 0 & 0 & 0 & 1 \\
 0 & 0 & 0 & 1 & 1 & 0 \\
 0 & 0 & 1 & 1 & 2 & 1 \\
 0 & 1 & 1 & 3 & 3 & 2 \\
 0 & 1 & 2 & 3 & 4 & 2 \\
 1 & 0 & 1 & 2 & 2 & 1 \\
\end{array}
\right) \\
\intertext{\normalsize and}
&
\left(
\begin{array}{cccccccc}
 1 & 0 & 0 & 0 & 0 & 0 & 0 & 0 \\
 0 & 1 & 0 & 0 & 0 & 0 & 0 & 0 \\
 0 & 0 & 1 & 0 & 0 & 0 & 0 & 0 \\
 0 & 0 & 0 & 1 & 0 & 0 & 0 & 0 \\
 0 & 0 & 0 & 0 & 1 & 0 & 0 & 0 \\
 0 & 0 & 0 & 0 & 0 & 1 & 0 & 0 \\
 0 & 0 & 0 & 0 & 0 & 0 & 1 & 0 \\
 0 & 0 & 0 & 0 & 0 & 0 & 0 & 1 \\
\end{array}
\right), 
\left(
\begin{array}{cccccccc}
 0 & 1 & 0 & 0 & 0 & 0 & 0 & 0 \\
 1 & 1 & 1 & 0 & 0 & 0 & 0 & 0 \\
 0 & 1 & 1 & 1 & 0 & 0 & 0 & 0 \\
 0 & 0 & 1 & 1 & 1 & 1 & 0 & 0 \\
 0 & 0 & 0 & 1 & 1 & 1 & 0 & 1 \\
 0 & 0 & 0 & 1 & 1 & 1 & 1 & 0 \\
 0 & 0 & 0 & 0 & 0 & 1 & 0 & 0 \\
 0 & 0 & 0 & 0 & 1 & 0 & 0 & 0 \\
\end{array}
\right), 
\left(
\begin{array}{cccccccc}
 0 & 0 & 1 & 0 & 0 & 0 & 0 & 0 \\
 0 & 1 & 1 & 1 & 0 & 0 & 0 & 0 \\
 1 & 1 & 1 & 1 & 1 & 1 & 0 & 0 \\
 0 & 1 & 1 & 2 & 2 & 2 & 1 & 1 \\
 0 & 0 & 1 & 2 & 2 & 2 & 1 & 0 \\
 0 & 0 & 1 & 2 & 2 & 2 & 0 & 1 \\
 0 & 0 & 0 & 1 & 1 & 0 & 0 & 0 \\
 0 & 0 & 0 & 1 & 0 & 1 & 0 & 0 \\
\end{array}
\right), 
\left(
\begin{array}{cccccccc}
 0 & 0 & 0 & 1 & 0 & 0 & 0 & 0 \\
 0 & 0 & 1 & 1 & 1 & 1 & 0 & 0 \\
 0 & 1 & 1 & 2 & 2 & 2 & 1 & 1 \\
 1 & 1 & 2 & 4 & 4 & 4 & 1 & 1 \\
 0 & 1 & 2 & 4 & 3 & 4 & 1 & 1 \\
 0 & 1 & 2 & 4 & 4 & 3 & 1 & 1 \\
 0 & 0 & 1 & 1 & 1 & 1 & 0 & 1 \\
 0 & 0 & 1 & 1 & 1 & 1 & 1 & 0 \\
\end{array}
\right), \\ & \qquad\qquad
\left(
\begin{array}{cccccccc}
 0 & 0 & 0 & 0 & 1 & 0 & 0 & 0 \\
 0 & 0 & 0 & 1 & 1 & 1 & 1 & 0 \\
 0 & 0 & 1 & 2 & 2 & 2 & 0 & 1 \\
 0 & 1 & 2 & 4 & 3 & 4 & 1 & 1 \\
 1 & 1 & 2 & 3 & 4 & 3 & 1 & 1 \\
 0 & 1 & 2 & 4 & 3 & 3 & 1 & 1 \\
 0 & 1 & 0 & 1 & 1 & 1 & 1 & 0 \\
 0 & 0 & 1 & 1 & 1 & 1 & 0 & 0 \\
\end{array}
\right), 
\left(
\begin{array}{cccccccc}
 0 & 0 & 0 & 0 & 0 & 1 & 0 & 0 \\
 0 & 0 & 0 & 1 & 1 & 1 & 0 & 1 \\
 0 & 0 & 1 & 2 & 2 & 2 & 1 & 0 \\
 0 & 1 & 2 & 4 & 4 & 3 & 1 & 1 \\
 0 & 1 & 2 & 4 & 3 & 3 & 1 & 1 \\
 1 & 1 & 2 & 3 & 3 & 4 & 1 & 1 \\
 0 & 0 & 1 & 1 & 1 & 1 & 0 & 0 \\
 0 & 1 & 0 & 1 & 1 & 1 & 0 & 1 \\
\end{array}
\right), 
\left(
\begin{array}{cccccccc}
 0 & 0 & 0 & 0 & 0 & 0 & 1 & 0 \\
 0 & 0 & 0 & 0 & 1 & 0 & 0 & 0 \\
 0 & 0 & 0 & 1 & 0 & 1 & 0 & 0 \\
 0 & 0 & 1 & 1 & 1 & 1 & 0 & 1 \\
 0 & 0 & 1 & 1 & 1 & 1 & 0 & 0 \\
 0 & 1 & 0 & 1 & 1 & 1 & 1 & 0 \\
 0 & 0 & 0 & 1 & 0 & 0 & 0 & 0 \\
 1 & 0 & 0 & 0 & 1 & 0 & 0 & 0 \\
\end{array}
\right), 
\left(
\begin{array}{cccccccc}
 0 & 0 & 0 & 0 & 0 & 0 & 0 & 1 \\
 0 & 0 & 0 & 0 & 0 & 1 & 0 & 0 \\
 0 & 0 & 0 & 1 & 1 & 0 & 0 & 0 \\
 0 & 0 & 1 & 1 & 1 & 1 & 1 & 0 \\
 0 & 1 & 0 & 1 & 1 & 1 & 0 & 1 \\
 0 & 0 & 1 & 1 & 1 & 1 & 0 & 0 \\
 1 & 0 & 0 & 0 & 0 & 1 & 0 & 0 \\
 0 & 0 & 0 & 1 & 0 & 0 & 0 & 0 \\
\end{array}
\right).
\end{align*}
}%
(Here, the $j$, $k$ entry of the $i$-th matrix gives the multiplicity of $X_k$ in $X_i X_j$.)
The bimodule category between the even parts has left module structure
{
\setlength\arraycolsep{3pt}
\scriptsize
\begin{align*}
&\left(
\begin{array}{cccccc}
 1 & 0 & 0 & 0 & 0 & 0 \\
 0 & 1 & 0 & 0 & 0 & 0 \\
 0 & 0 & 1 & 0 & 0 & 0 \\
 0 & 0 & 0 & 1 & 0 & 0 \\
 0 & 0 & 0 & 0 & 1 & 0 \\
 0 & 0 & 0 & 0 & 0 & 1 \\
\end{array}
\right), 
\left(
\begin{array}{cccccc}
 1 & 1 & 0 & 0 & 0 & 0 \\
 1 & 1 & 1 & 0 & 0 & 0 \\
 0 & 1 & 1 & 1 & 0 & 0 \\
 0 & 0 & 1 & 2 & 1 & 1 \\
 0 & 0 & 0 & 1 & 1 & 0 \\
 0 & 0 & 0 & 1 & 0 & 1 \\
\end{array}
\right), 
\left(
\begin{array}{cccccc}
 0 & 1 & 1 & 0 & 0 & 0 \\
 1 & 1 & 1 & 1 & 0 & 0 \\
 1 & 1 & 1 & 2 & 1 & 1 \\
 0 & 1 & 2 & 4 & 2 & 2 \\
 0 & 0 & 1 & 2 & 0 & 1 \\
 0 & 0 & 1 & 2 & 1 & 0 \\
\end{array}
\right), \\ & \qquad\qquad
\left(
\begin{array}{cccccc}
 0 & 0 & 1 & 1 & 0 & 0 \\
 0 & 1 & 1 & 2 & 1 & 1 \\
 1 & 1 & 2 & 4 & 2 & 2 \\
 1 & 2 & 4 & 8 & 3 & 3 \\
 0 & 1 & 2 & 3 & 1 & 2 \\
 0 & 1 & 2 & 3 & 2 & 1 \\
\end{array}
\right), 
\left(
\begin{array}{cccccc}
 0 & 0 & 0 & 1 & 1 & 1 \\
 0 & 0 & 1 & 3 & 1 & 1 \\
 0 & 1 & 3 & 5 & 2 & 2 \\
 1 & 3 & 5 & 9 & 4 & 4 \\
 1 & 1 & 2 & 4 & 2 & 1 \\
 1 & 1 & 2 & 4 & 1 & 2 \\
\end{array}
\right), 
\left(
\begin{array}{cccccc}
 0 & 0 & 0 & 1 & 0 & 0 \\
 0 & 0 & 1 & 1 & 1 & 1 \\
 0 & 1 & 1 & 3 & 1 & 1 \\
 1 & 1 & 3 & 5 & 2 & 2 \\
 0 & 1 & 1 & 2 & 1 & 1 \\
 0 & 1 & 1 & 2 & 1 & 1 \\
\end{array}
\right)
\end{align*}
}
and right module structure
{
\setlength\arraycolsep{3pt}
\scriptsize
\begin{align*}
&\left(
\begin{array}{cccccc}
 1 & 0 & 0 & 0 & 0 & 0 \\
 1 & 1 & 0 & 0 & 0 & 0 \\
 0 & 1 & 1 & 0 & 0 & 0 \\
 0 & 0 & 1 & 1 & 0 & 0 \\
 0 & 0 & 0 & 1 & 1 & 0 \\
 0 & 0 & 0 & 1 & 0 & 1 \\
 0 & 0 & 0 & 0 & 0 & 1 \\
 0 & 0 & 0 & 0 & 1 & 0 \\
\end{array}
\right), 
\left(
\begin{array}{cccccc}
 0 & 1 & 0 & 0 & 0 & 0 \\
 1 & 1 & 1 & 0 & 0 & 0 \\
 1 & 1 & 1 & 1 & 0 & 0 \\
 0 & 1 & 1 & 2 & 1 & 1 \\
 0 & 0 & 1 & 2 & 1 & 1 \\
 0 & 0 & 1 & 2 & 1 & 1 \\
 0 & 0 & 0 & 1 & 0 & 0 \\
 0 & 0 & 0 & 1 & 0 & 0 \\
\end{array}
\right), 
\left(
\begin{array}{cccccc}
 0 & 0 & 1 & 0 & 0 & 0 \\
 0 & 1 & 1 & 1 & 0 & 0 \\
 1 & 1 & 1 & 2 & 1 & 1 \\
 1 & 1 & 2 & 4 & 2 & 2 \\
 0 & 1 & 2 & 4 & 1 & 2 \\
 0 & 1 & 2 & 4 & 2 & 1 \\
 0 & 0 & 1 & 1 & 1 & 0 \\
 0 & 0 & 1 & 1 & 0 & 1 \\
\end{array}
\right), \\ & \qquad\qquad
\left(
\begin{array}{cccccc}
 0 & 0 & 0 & 1 & 0 & 0 \\
 0 & 0 & 1 & 2 & 1 & 1 \\
 0 & 1 & 2 & 4 & 2 & 2 \\
 1 & 2 & 4 & 8 & 3 & 3 \\
 1 & 2 & 4 & 7 & 3 & 3 \\
 1 & 2 & 4 & 7 & 3 & 3 \\
 0 & 1 & 1 & 2 & 1 & 1 \\
 0 & 1 & 1 & 2 & 1 & 1 \\
\end{array}
\right), 
\left(
\begin{array}{cccccc}
 0 & 0 & 0 & 0 & 1 & 0 \\
 0 & 0 & 0 & 1 & 0 & 1 \\
 0 & 0 & 1 & 2 & 1 & 0 \\
 0 & 1 & 2 & 3 & 1 & 2 \\
 1 & 1 & 1 & 3 & 2 & 1 \\
 0 & 1 & 2 & 3 & 1 & 1 \\
 1 & 0 & 0 & 1 & 0 & 1 \\
 0 & 0 & 1 & 1 & 0 & 0 \\
\end{array}
\right), 
\left(
\begin{array}{cccccc}
 0 & 0 & 0 & 0 & 0 & 1 \\
 0 & 0 & 0 & 1 & 1 & 0 \\
 0 & 0 & 1 & 2 & 0 & 1 \\
 0 & 1 & 2 & 3 & 2 & 1 \\
 0 & 1 & 2 & 3 & 1 & 1 \\
 1 & 1 & 1 & 3 & 1 & 2 \\
 0 & 0 & 1 & 1 & 0 & 0 \\
 1 & 0 & 0 & 1 & 1 & 0 \\
\end{array}
\right).
\end{align*}
}

From this, we calculate
\begin{equation*}
\cM = \left(
\begin{array}{cccccccccccccc}
 6 & 5 & 8 & 13 & 15 & 9 & 6 & 5 & 8 & 13 & 12 & 12 & 3 & 3 \\
 5 & 19 & 26 & 45 & 52 & 28 & 7 & 17 & 28 & 43 & 41 & 41 & 11 & 11 \\
 8 & 26 & 56 & 93 & 110 & 60 & 6 & 28 & 54 & 95 & 84 & 84 & 26 & 26 \\
 13 & 45 & 93 & 181 & 211 & 115 & 13 & 45 & 93 & 181 & 163 & 163 & 48 & 48 \\
 15 & 52 & 110 & 211 & 259 & 138 & 16 & 51 & 111 & 210 & 199 & 199 & 60 & 60 \\
 9 & 28 & 60 & 115 & 138 & 79 & 8 & 29 & 59 & 116 & 108 & 108 & 30 & 30 \\
 6 & 7 & 6 & 13 & 16 & 8 & 8 & 5 & 8 & 11 & 13 & 13 & 3 & 3 \\
 5 & 17 & 28 & 45 & 51 & 29 & 5 & 17 & 28 & 45 & 40 & 40 & 11 & 11 \\
 8 & 28 & 54 & 93 & 111 & 59 & 8 & 28 & 54 & 93 & 85 & 85 & 26 & 26 \\
 13 & 43 & 95 & 181 & 210 & 116 & 11 & 45 & 93 & 183 & 162 & 162 & 48 & 48 \\
 12 & 41 & 84 & 163 & 199 & 108 & 13 & 40 & 85 & 162 & 154 & 154 & 45 & 45 \\
 12 & 41 & 84 & 163 & 199 & 108 & 13 & 40 & 85 & 162 & 154 & 154 & 45 & 45 \\
 3 & 11 & 26 & 48 & 60 & 30 & 3 & 11 & 26 & 48 & 45 & 45 & 15 & 15 \\
 3 & 11 & 26 & 48 & 60 & 30 & 3 & 11 & 26 & 48 & 45 & 45 & 15 & 15 \\
\end{array}
\right)
\end{equation*}
and find $$\cD = 50 \zeta _{13}^{11}+50 \zeta _{13}^{10}-125 \zeta _{13}^9-125 \zeta _{13}^7-125 \zeta _{13}^6-125 \zeta _{13}^4+50 \zeta _{13}^3+50 \zeta _{13}^2+170,$$
and
\begin{align*}
v_{EH1}  = \Big\{&1, \\
& \zeta _{13}^{11}+\zeta _{13}^{10}+\zeta _{13}^3+\zeta _{13}^2+2, \\
& \zeta _{13}^{11}+\zeta _{13}^{10}-\zeta _{13}^9-\zeta _{13}^7-\zeta _{13}^6-\zeta _{13}^4+\zeta _{13}^3+\zeta _{13}^2+3, \displaybreak[1]\\
& \zeta _{13}^{11}+\zeta _{13}^{10}-3 \zeta _{13}^9-3 \zeta _{13}^7-3 \zeta _{13}^6-3 \zeta _{13}^4+\zeta _{13}^3+\zeta _{13}^2+4, \\
& \zeta _{13}^{11}+\zeta _{13}^{10}-4 \zeta _{13}^9-4 \zeta _{13}^7-4 \zeta _{13}^6-4 \zeta _{13}^4+\zeta _{13}^3+\zeta _{13}^2+4, \displaybreak[1]\\
& \zeta _{13}^{11}+\zeta _{13}^{10}-2 \zeta _{13}^9-2 \zeta _{13}^7-2 \zeta _{13}^6-2 \zeta _{13}^4+\zeta _{13}^3+\zeta _{13}^2+2, \\
& 0,  0,  0,  0,  0,  0,  0, 0 \Big\}\\
v_{EH2}  = \Big\{&0, 0, 0, 0, 0, 0, 1, \\
& \zeta _{13}^{11}+\zeta _{13}^{10}+\zeta _{13}^3+\zeta _{13}^2+2, \displaybreak[1]\\
& \zeta _{13}^{11}+\zeta _{13}^{10}-\zeta _{13}^9-\zeta _{13}^7-\zeta _{13}^6-\zeta _{13}^4+\zeta _{13}^3+\zeta _{13}^2+3, \\
& \zeta _{13}^{11}+\zeta _{13}^{10}-3 \zeta _{13}^9-3 \zeta _{13}^7-3 \zeta _{13}^6-3 \zeta _{13}^4+\zeta _{13}^3+\zeta _{13}^2+4, \displaybreak[1]\\
& \zeta _{13}^{11}+\zeta _{13}^{10}-3 \zeta _{13}^9-3 \zeta _{13}^7-3 \zeta _{13}^6-3 \zeta _{13}^4+\zeta _{13}^3+\zeta _{13}^2+3, \\
& \zeta _{13}^{11}+\zeta _{13}^{10}-3 \zeta _{13}^9-3 \zeta _{13}^7-3 \zeta _{13}^6-3 \zeta _{13}^4+\zeta _{13}^3+\zeta _{13}^2+3, \displaybreak[1]\\
& -\zeta _{13}^9-\zeta _{13}^7-\zeta _{13}^6-\zeta _{13}^4+1, \\
& -\zeta _{13}^9-\zeta _{13}^7-\zeta _{13}^6-\zeta _{13}^4+1
 \Big\}
\end{align*}
where $\zeta_{13} = \exp(2\pi i/13)$ is a primitive $13$-th root of unity.

This $\cM$ has rank $6$. However its leading $6$-by-$6$ minor is singular; we need to need permute the rows and columns before we can apply Theorem \ref{thm:reduce}. In fact, the computational difficulty of the subsequent calculations depends on the choice made here. The rule of thumb we use is to first permute rows and columns so that the diagonal entries are increasing, and then take the lexicographically least 6 element subset of the rows and columns so that the corresponding minor is non-singular.

We obtain
\begin{equation*}
\cM' = \left(
\begin{array}{cccccc}
 6 & 6 & 3 & 5 & 9 & 13 \\
 6 & 8 & 3 & 5 & 8 & 13 \\
 3 & 3 & 15 & 11 & 30 & 48 \\
 5 & 5 & 11 & 17 & 29 & 45 \\
 9 & 8 & 30 & 29 & 79 & 115 \\
 13 & 13 & 48 & 45 & 115 & 181 \\
\end{array}
\right)
\end{equation*}
with
\begin{equation*}
\cR = \left(
\begin{array}{cccccc}
 1 & 0 & 0 & 0 & 0 & 0 \\
 -1 & 1 & 0 & 1 & 0 & 0 \\
 1 & -1 & 1 & 1 & 0 & 0 \\
 0 & 0 & 0 & 0 & 0 & 1 \\
 -1 & 1 & 2 & 0 & 1 & 0 \\
 0 & 0 & 0 & 0 & 1 & 0 \\
 0 & 1 & 0 & 0 & 0 & 0 \\
 0 & 0 & 0 & 1 & 0 & 0 \\
 0 & 0 & 1 & 1 & 0 & 0 \\
 1 & -1 & 0 & 0 & 0 & 1 \\
 -1 & 1 & 1 & 0 & 1 & 0 \\
 -1 & 1 & 1 & 0 & 1 & 0 \\
 0 & 0 & 1 & 0 & 0 & 0 \\
 0 & 0 & 1 & 0 & 0 & 0 \\
\end{array}
\right)
\end{equation*}

In about 2 minutes of computer time, we find that $\cM'$ has a unique algebraic decomposition. Preparing the corresponding unique algebraic decomposition of $\cM$ according to Theorem \ref{thm:reduce}, we find the combinatorial induction functors for the extended Haagerup subfactor given on the first page of this article.

This method also uniquely finds the combinatorial data of the induction functor for the Haagerup subfactor (which has appeared already in \cite{MR1832764}) and for the Asaeda-Haagerup subfactor, where it is given by

\begin{align*}
I_{AH1} & = \left(
\begin{array}{cccccccccccccccccccccc}
 1 & 1 & 0 & 0 & 0 & 1 & 1 & 1 & 1 & 0 & 0 & 0 & 0 & 0 & 0 & 0 & 0 & 0 & 0 & 0 & 0 & 0 \\
 1 & 0 & 1 & 1 & 1 & 2 & 1 & 1 & 0 & 1 & 1 & 1 & 1 & 1 & 1 & 1 & 1 & 1 & 0 & 0 & 0 & 0 \\
 1 & 2 & 1 & 3 & 1 & 1 & 2 & 2 & 0 & 1 & 2 & 2 & 2 & 2 & 2 & 2 & 2 & 2 & 2 & 2 & 0 & 2 \\
 0 & 1 & 2 & 2 & 1 & 1 & 1 & 2 & 0 & 1 & 1 & 1 & 1 & 1 & 1 & 1 & 1 & 1 & 2 & 2 & 1 & 1 \\
 1 & 3 & 3 & 2 & 1 & 1 & 3 & 3 & 0 & 1 & 3 & 3 & 3 & 3 & 3 & 3 & 3 & 3 & 3 & 3 & 2 & 1 \\
 2 & 2 & 2 & 1 & 1 & 0 & 1 & 0 & 0 & 1 & 1 & 1 & 1 & 1 & 1 & 1 & 1 & 1 & 1 & 1 & 1 & 0 \\
\end{array}
\right) \displaybreak[1]\\ 
I_{AH2} & = \left(
\begin{array}{cccccccccccccccccccccc}
 1 & 0 & 0 & 0 & 0 & 2 & 1 & 1 & 1 & 1 & 0 & 0 & 0 & 0 & 0 & 0 & 0 & 0 & 0 & 0 & 0 & 0 \\
 1 & 1 & 1 & 1 & 1 & 1 & 1 & 1 & 0 & 0 & 1 & 1 & 1 & 1 & 1 & 1 & 1 & 1 & 0 & 0 & 0 & 0 \\
 1 & 1 & 1 & 3 & 1 & 2 & 2 & 2 & 0 & 2 & 2 & 2 & 2 & 2 & 2 & 2 & 2 & 2 & 2 & 2 & 0 & 2 \\
 0 & 2 & 2 & 2 & 1 & 1 & 2 & 3 & 0 & 0 & 2 & 2 & 2 & 2 & 2 & 2 & 2 & 2 & 2 & 2 & 1 & 1 \\
 1 & 3 & 3 & 2 & 1 & 0 & 2 & 2 & 0 & 1 & 2 & 2 & 2 & 2 & 2 & 2 & 2 & 2 & 3 & 3 & 2 & 1 \\
 1 & 1 & 1 & 0 & 0 & 0 & 1 & 0 & 0 & 1 & 1 & 1 & 1 & 1 & 1 & 1 & 1 & 1 & 1 & 1 & 1 & 0 \\
 1 & 1 & 1 & 0 & 0 & 0 & 1 & 0 & 0 & 1 & 1 & 1 & 1 & 1 & 1 & 1 & 1 & 1 & 1 & 1 & 1 & 0 \\
 1 & 1 & 1 & 1 & 1 & 1 & 1 & 1 & 0 & 0 & 1 & 1 & 1 & 1 & 1 & 1 & 1 & 1 & 0 & 0 & 0 & 0 \\
 1 & 1 & 1 & 1 & 1 & 0 & 0 & 0 & 0 & 0 & 0 & 0 & 0 & 0 & 0 & 0 & 0 & 0 & 0 & 0 & 0 & 0 \\
\end{array}
\right)
\end{align*}
For some fusion categories, however, for example the even part of the 4442 subfactor described in \cite{1208.3637}, this method seems to be insufficient, or to at least require a faster implementation. (We stopped the search after a day of computer time.)

This approach to computing the combinatorial induction functor does not appear to be useful for ruling out candidate fusion rings; so far we haven't found an interesting example. In fact, often the method does not produce a unique answer. A simple example is for the principal graphs
$$\left(\bigraph{bwd1v4v1duals1v1},\bigraph{bwd1v4v1duals1v1}\right)$$
which have a unique compatible fusion ring, but four different compatible combinatorial induction functors, with either 4, 6, 7, or 12 simple objects in the centre.

\newcommand{\notfree}{}
\bibliographystyle{alpha}
\bibliography{../../bibliography/bibliography}

\end{document}